\documentclass[oneside]{amsart}
\usepackage{newcent}
\usepackage{verbatim}
\usepackage{amsthm}
\usepackage{amssymb}

\makeatletter
\numberwithin{equation}{section} 
\numberwithin{figure}{section} 
\theoremstyle{plain}
\newtheorem{thm}{Theorem}[section]
  \theoremstyle{remark}
  \newtheorem*{acknowledgement*}{Acknowledgement}
  \theoremstyle{plain}
  \newtheorem{lem}[thm]{Lemma}
  \theoremstyle{definition}
  \newtheorem{defn}[thm]{Definition}
  \theoremstyle{plain}
  \newtheorem{prop}[thm]{Proposition}
  \theoremstyle{plain}
  \newtheorem{cor}[thm]{Corollary}
  \theoremstyle{remark}
  \newtheorem{rem}[thm]{Remark}
  \theoremstyle{remark}
  \newtheorem*{rem*}{Remark}

\DeclareMathOperator{\dom}{dom}
\DeclareMathOperator{\re}{Re}
\DeclareMathOperator{\supp}{supp}

\newcommand{\E}{\ensuremath{E}}
\newcommand{\D}{\ensuremath{D}}

\makeatother

\begin{document}

\title[Complex Powers of the Laplacian on Affine Nested Fractals]{Complex Powers of the Laplacian on Affine Nested Fractals as Calder\'on-Zygmund
operators}


\author{Marius Ionescu}
\address{Department of Mathematics, Colgate University,  Hamilton, NY, 13346, USA.}
\email{mionescu@colgate.edu}

\author{Luke Rogers}
\address{Department of Mathematics, University of Connecticut, Storrs, CT,
06269, USA.}
\email{rogers@math.uconn.edu}

\thanks{This research was partially supported by a grant from
the Simons Foundation (\#209277 to Marius Ionescu).}

\begin{abstract}
We give the first natural examples of Calder\'on-Zygmund operators in the theory
of analysis on post-critically finite self-similar fractals.  This is achieved
by showing that the purely imaginary Riesz and Bessel potentials on nested
fractals with $3$ or more boundary points are of this type.  It follows that
these operators are bounded on $L^{p}$, $1<p<\infty$ and satisfy weak $1$-$1$
bounds.  The analysis may be extended to infinite blow-ups of these fractals,
and to product spaces based on the fractal or its blow-up.
\end{abstract}

\maketitle

\section{Introduction}

Complex powers of the Laplacian on Euclidean spaces and manifolds
and their connection to pseudodifferential operators have been studied
intensely (see, for example, \cite{See_67,See_69,Ste_AMS63_70,Tay_PMS81,Dav_CTM90}
and the citations within). In this paper we define and study a class
of operators built from the Laplace operator $\Delta$ on nested
fractals~\cite{Lin_MAMS90}, which are a type of post-critically finite
self-similar fractal~\cite{Kig_CUP01,Str_Prin06}. 
The main focus is to show that the Riesz potentials $(-\Delta)^{i\alpha}$
and the Bessel potentials $(I-\Delta)^{i\alpha}$,
$\alpha\in\mathbb{R}\setminus \{0\}$,
are Calder\'on-Zygmund operators in the sense of \cite{Ste_PMS43_93}.
These operators are the first explicit examples of Calder\'on-Zygmund
operators on a general class of self-similar fractals. The main result is as follows.
\begin{thm}
Let $K$ be a nested fractal and $X$ be either $K$ or an infinite blow-up
of $K$ without boundary.  Suppose $T$ is an bounded operator on $L^{2}(\mu)$, where $\mu$ is the self-similar measure on $X$,
and there is a kernel $K(x,y)$ such that
\[ T(f)(x)=\int_{X}K(x,y)f(y)d\mu(y) \]
for $f\in L^{2}(\mu)$ and almost all $x\notin\supp f$. If $K(x,y)$ is a smooth
function off the diagonal of $X\times X$ and satisfies
\begin{align}
    \bigl| K(x,y)\bigr|
    &\lesssim R(x,y)^{-d} \label{eq:intro_est1}\\
    \bigl|\Delta_{2}K(x,y)\bigr|
    &\lesssim R(x,y)^{-2d-1},\label{eq:intro_est2}
    \end{align}
where $R(x,y)$ is the resistance metric on $X$ and $d$ is the dimension
of $X$ with respect to the resistance metric, then the operator $T$
is a Calder\'on-Zygmund operator in the sense of~\cite[Section I.6.5]{Ste_PMS43_93}.
\end{thm}
This is proved as Theorem \ref{thm:Singular_integral}. In Sections
\ref{sec:Imaginary-Powers} and \ref{sec:Bessel-Potentials} we show
that the Riesz and Bessel potentials have kernels which satisfy conditions
\eqref{eq:intro_est1} and \eqref{eq:intro_est2}. We conclude
that these potentials are Calder\'on-Zygmund operators and that
they are bounded on $L^{p}(X)$, $1<p<\infty$. Furthermore, we study the general
Bessel operators $(I-\Delta)^{\alpha}$, $\alpha\in\mathbb{C}$ and
prove that, for $\re\alpha<0$, they are given by integration with
respect to kernels which are smooth off the diagonal and they are
bounded on $L^{p}(X)$. In Section~\ref{sec:Products} we extend our
analysis to products of nested fractals and their infinite blowups.

Riesz and Bessel potentials for negative real powers in the context
of metric measure spaces, including fractals, have been studied
in~\cite{HuZa_09} (see also~\cite{HuZa_05}), however their results are not
directly applicable in our setting. The main tool we use is estimates of kernels
of the form~\eqref{laplace_kernel} and powers of the Laplacian applied to these
kernels. Our estimates
are closely related to those in Section~4 of~\cite{HuZa_09}, where integrals
obtained by replacing $m(t)dt$ in~\eqref{laplace_kernel} by a non-negative measure $d\nu$ are treated.
Our results are not contained in theirs, because we are primarily interested in
complex-valued oscillatory functions $m(t)$ and need estimates of Laplacians of
the kernel.  Nor are their results contained in ours, because their work can
apply to singular measures and measures for which the absolutely continuous part
is not bounded.

\begin{acknowledgement*}
The authors would like to thank Robert S. Strichartz for numerous
and fruitful discussions during our work on this project.
\end{acknowledgement*}

\section{Background}\label{backgroundnotation}

In this paper $K$ denotes a nested fractal in the sense of
Lindstr\"{o}m~\cite{Lin_MAMS90}.  These are a subclass of the post-critically
finite self-similar fractals, on which there is an analytic theory due to
Kigami~\cite{Kig_CUP01}.  We will find it convenient to use the notation and
constructions of Kigami to describe the analytic structure on $K$, rather than
the equivalent probabilistic construction made by Lindstr\"{o}m, and we only
include enough information here to provide notation for our later results.
Further details and proofs are in~\cite{Kig_CUP01}.

\subsection*{Nested fractals: Energy, Laplacian, smoothness, resistance metric}
An iterated function system (i.f.s.) is a collection $\{F_{1},\dots,F_{N}\}$
of contractions on $\mathbb{R}^{d}$. For such an i.f.s. there exists
a unique invariant set $K$ satisfying (see~\cite{Hut_81})
\[
K=F_{1}(K)\bigcup\dots\bigcup F_{N}(K).\]
$K$ is called the self-similar set associated to the i.f.s. We assume
in this paper that $\{F_{1},\dots,F_{N}\}$ are contractive similitudes
satisfying the open set condition. That is, there is a dense open subset
$O\subset K$ such that $F_i(O)\bigcap F_j(O)=\emptyset$ if $i\ne j$. For $\omega_{1},\dotsc,\omega_{n}\in\{1,\dots,N\}$,
$\omega=\omega_{1}\dotsm\omega_{n}$ is a word of length
$n$ over the alphabet $\{1,\dots,N\}$. Then $K_{\omega}=F_{\omega}(K):=F_{\omega_{1}}\circ\dots\circ F_{\omega_{n}}(K)$
is called a cell of level $n$. The set of all finite words over $\{1,\dots,N\}$
is denoted by $W_{*}$. Each map $F_{i}$ of the i.f.s. defining $K$
has a unique fixed point $x_{i}$. Then $K$ is a {\em post-critically
finite (PCF) self-similar} set if there is a subset $V_{0}\subseteq\{x_{1},\dots,x_{N}\}$
satisfying
\[
F_{\omega}(K)\bigcap F_{\omega^{\prime}}(K)\subseteq F_{\omega}(V_{0})\bigcap F_{\omega^{\prime}}(V_{0})\]
for any $\omega\ne\omega^{\prime}$ having the same length. The set
$V_{0}$ is called the \emph{boundary} of $K$ and the boundary of
a cell $K_{\omega}$ is $F_{\omega}(V_{0})$. We define $V_{1}=\bigcup_{i}F_{i}(V_{0})$,
and, inductively, $V_{n}=\bigcup_{i}F_{i}(V_{n-1})$ for $n\ge2$.

It is well known (see~\cite{Hut_81}) that for weights $\{\mu_{1},\dots,\mu_{N}\}$ such that $0<\mu_{i}<1$ there is a unique self-similar measure
\[ \mu(A)=\sum_{i=1}^{N}\mu_{i}\mu(F^{-1}(A)).\]

A nested fractal is of the above type, but in addition has a large symmetry
group.  For $K$ to be a nested fractal one requires that for every pair of
points $p,q\in V_{0}$ the reflection in the Euclidean hyperplane equidistant
from $p$ and $q$ maps $n$ cells to $n$ cells.  It is also required that any $n$
cell that intersects the hyperplane at a non-boundary point (of the $n$ cell) is
mapped to itself by the reflection.  For full details see~\cite{Lin_MAMS90}.

A key feature of nested fractals is the existence of a regular self-similar Dirichlet energy form $\E$
on $K$ with weights $0<r_{i}<1$, $i=1,\dotsc,N$ such that
\[ \E(u)=\sum_{i=1}^{N}r_{i}^{-1}\E(u\circ F_{i}).\]
The existence of such forms is non-trivial.  In the case that all $r_{i}$ are
the same it is due to Lindstr\"{o}m~\cite{Lin_MAMS90} via probabilistic methods.
Kigami's approach to constructing these as limits of resistance forms may be
found in~\cite{Kig_CUP01,Str_Prin06}.  When the $r_{i}$ are not all equal there
is no known general solution.

Let $u\in\dom\E$ and $f$ be continuous on $X$.  We say
$u\in\dom\Delta$ with (Neumann) Laplacian $\Delta u=f$ if
\[\E(u,v)=-\int_{X}fv\,d\mu\]
for all $v\in\dom\E$. We say that $u$ is
\emph{smooth} if $\Delta^n u$ is continuous for all $n\ge 1$.  The operator $-\Delta$ is non-negative definite and self-adjoint, with
eigenvalues $0=\lambda_{1}\le\lambda_{2}\le\dots$ accumulating only at $\infty$.
We fix an orthonormal basis $\{\varphi_{n}\}$ for $L^{2}(\mu)$ where
$\varphi_{n}$ has eigenvalue $\lambda_{n}$ and eigenvalues may be repeated and
let $\D$ be the set of finite linear combinations of $\varphi_{n}$.

The \emph{effective resistance metric} $R(x,y)$ on $K$ is defined by
\[ R(x,y)^{-1}=\min\{\E(u)\,:\, u(x)=0\text{ and }u(x)=1\}.\]
It is known that the resistance metric is topologically equivalent,
but not metrically equivalent to the Euclidean metric~\cite{Kig_CUP01,Str_Prin06}).

\subsection*{Examples}
An important example of a PCF self-similar set is the unit interval
$I=[0,1]$. In this case $V_{0}=\{0,1\}$. While $I$ is not a fractal,
Kigami's construction applies and one recovers the usual energy on
the interval\[
\E(u,v)=\int_{0}^{1}u^{\prime}(x)v^{\prime}(x)dx,\]
and the usual Laplacian $\Delta u=u^{\prime\prime}$.

The simplest example of a fractal to which the theory applies is the Sierpinski gasket, which has been studied
intensively (see, for example, \cite{Kig_CUP01,Str_Prin06,FuSh_92,BassStrTep_JFA99,NeStTe_04,Str_TAMS03,Tep_JFA98}).
To describe the Sierpinski gasket, consider a triangle in $\mathbb{R}^{2}$
with vertices $\{q_{0},q_{1},q_{2}\}$ and consider a set of three
mappings $F_{i}:\mathbb{R}^{2}\to\mathbb{R}^{2}$, $i=1,2,3$, defined
by \[
F_{i}(x)=\frac{1}{2}(x-q_{i})+q_{i}.\]
The invariant set of this iterated function system is the Sierpinski
gasket and $V_{0}=\{q_{0},q_{1},q_{2}\}$.

\subsection*{Blow-ups}
In \cite{Str_CJM98,Str_TAMS03} Strichartz defined fractal blow-ups of $K$.  This construction generalizes the relationship between the unit interval and the real line to arbitrary PCF self-similar sets.

Let $w\in\{1,\dots,N\}^{\infty}$
be an infinite word. Then \[
F_{w_{1}}^{-1}\dots F_{w_{m}}^{-1}K\subseteq F_{w_{1}}^{-1}\dots F_{w_{m}}^{-1}F_{w_{m+1}}^{-1}K.\]
The \emph{fractal blow-up} $K_{\infty}$ is\[
K_{\infty}=\bigcup_{m=1}^{\infty}F_{w_{1}}^{-1}\dots F_{w_{m}}^{-1}K.\]
If $C$ is an $n$ cell in $K$, then $F_{w_{1}}^{-1}\dots F_{w_{m}}^{-1}C$
is called an $(n-m)$ cell. The blow-up depends on the choice of the
infinite word $w$. In general there are an uncountably infinite number
of blow-ups which are not homeomorphic. In this paper we assume that
the infinite blow-up $K_{\infty}$ has no boundary. This happens unless
all but a finite number of letters in $w$ are the same. One can extend the
definition of the energy $\E$ and measure $\mu$ to $K_{\infty}$.
The measure $\mu$ will be $\sigma$-finite rather than finite.
As before, $\Delta$ is defined by $u\in\dom\Delta$ with $\Delta u=f$ if $u\in\dom\E$,
$f$ is continuous, and\[
\E(u,v)=-\int_{K_{\infty}}fvd\mu\]
for every $v\in\dom\E$.

It will be important in what follows that if $K$ is a nested fractal and $V_{0}$
contains $3$ or more points, then the Laplacian on an infinite blow-up without
boundary has pure point spectrum~\cite{Tep_JFA98,Sab:JFA00} and the
eigenfunctions have compact support.
We write $\{\lambda_{n}\}_{n\in\mathbb{Z}}$ for the eigenvalues of $-\Delta$, which are non-negative and accumulate only at $0$ and
$\infty$.  As for $K$, we take a basis $\{\varphi_{n}\}_{n\in\mathbb{Z}}$ of
$L^{2}(\mu)$ consisting of eigenfunctions of $-\Delta$ with eigenvalues
$\lambda_{n}$ that may be repeated, and let $\D$ be the set of finite linear
combinations of $\varphi_{n}$.

\subsection*{Notation for estimates}
We write $A(y)\lesssim B(y)$ if there is a constant
$C$ \emph{independent }of $y$, but which might depend on the fractal
$K$, such that $A(y)\le CB(y)$ for all $y$. We write $A(y)\sim B(y)$
if $A(y)\lesssim B(y)$ and $B(y)\lesssim A(y)$. If $f(x,y)$ is
a function on $X\times X$, then we write $\Delta_{1}f$ to denote
the Laplacian of $f$ with respect to the first variable and $\Delta_{2}f$
to denote the Laplacian of $f$ with respect to the second variable; repeated subscripts indicate composition, for example $\Delta_{21}=\Delta_{2}\circ\Delta_{1}$.

\subsection*{Heat kernel on nested fractals}

Let $K$ be an affine nested fractal with Dirichlet form as above.  Let $\mu$ be the unique self-similar probability measure for which $\mu_{i}=r_{i}^{d}$, so that $K$ is of Hausdorff dimension $d$ in the resistance metric, and let $X$
be an infinite blow-up of $K$ without boundary.
A fundamental result we require is an estimate for the heat kernel corresponding to the Laplacian. Specifically, the semi-group $e^{t\Delta}$ is given by integration with
respect to a positive heat kernel $h_{t}(x,y)$ which satisfies
\begin{equation}
h_{t}(x,y)\lesssim
t^{-\beta}\mbox{exp}\left(-c\left(\frac{R(x,y)^{d+1}}{t}\right)^{\gamma}\right)\text{
  for } 0<t<1,\label{eq:heat_estimates}
\end{equation}
where $\beta=d/(d+1)$, $R(x,y)$ is the effective resistance metric
on $X$, and $\gamma=\frac{\gamma'}{d+1-\gamma'}$, where $0<\gamma'<d+1$,
is the chemical exponent, a constant depending on the fractal. 
These estimates are originally due to Barlow and Perkins \cite{BaPe_PTRF88} for the case of the Sierpinski gasket
and have been generalized beyond what is needed here.  In particular,
\eqref{eq:heat_estimates} is a special case of~\cite[Theorem 1.1(2)
and Remark 3.7(2)]{FiHaKu_94}, but see also~\cite{HaKu_PLMS99}
and~\cite{HaKu_PTRD03}. Strichartz proved in \cite{Str_08Quant} that
if $X$ is an infinite blow-up of the Sierpinski gasket then the
estimate \eqref{eq:heat_estimates} holds for all $t\in(0,\infty)$. Moreover, lower estimates for the heat kernel
are proved in the papers mentioned above, but we will not use them
in this paper. We will need, however, the fact that the derivatives of
the heat kernel satisfy similar estimates to
\eqref{eq:heat_estimates}. Presumably this fact is known to
specialists but we have been unable to find a reference in the
literature, other than \cite[Proposition 7.5]{BaPe_PTRF88}. Their
estimates are related to what we need but they proved them only for the
Sierpinski gasket.
\begin{thm}\label{th:est1}
  Let $X$ be an affine nested fractal or an infinite blow-up of such a
  fractal. For $0<t<1$ we have that
  \begin{equation}
    \label{eq:heat_est_deriv}
    \left\vert t^k\left(\frac{\partial}{\partial t}\right)^k
      h_t(x,y)\right\vert \lesssim t^{-\beta}\exp\left(-c\left(\frac{R(x,y)^{d+1}}{t}\right)^{\gamma}\right).
  \end{equation}
\end{thm}
\begin{proof}
  Following the proof of Theorem 10.2 in \cite{Rog_08} we take a
  contour $\Gamma_t$ consisting of arc of the circle of radius $1/t$
  between angles $-\frac{3\pi}{4}$ and $\frac{3\pi}{4}$, together with
  rays $se^{\pm i\frac{3\pi}{4}}$, $s\in [1/t,\infty)$, and write $t^k\left(\frac{\partial}{\partial t}\right)^k
  h_t(x,y)$ using the resolvent kernel $G^{(z)}(x,y)$ to obtain
  \begin{eqnarray*}
    \label{eq:heat_rez}
    \left\vert t^k\left(\frac{\partial}{\partial t}\right)^k
      h_t(x,y)\right\vert & =& \left\vert \frac{1}{2\pi
        i}\int_{\Gamma_t} t^kz^ke^{zt}G^{(z)}(x,y)dz\right\vert \\
    &\le& \sup_{z\in \Gamma_t}\vert
    G^{(z)}(x,y)\vert\left(\int_{\Gamma_t}\vert t^kz^ke^{zt}\vert
      \vert dz\vert\right).
  \end{eqnarray*}
  But on the rays we have $\vert e^{zt}\vert\le e^{-\frac{\vert zt\vert}{\sqrt{2}}}$ so they contribute at
  most
  \[
    \int_{s\ge \frac{1}{t}}(st)^k
    e^{-\frac{st}{\sqrt{2}}}ds=\frac{2^{\frac{k+1}{2}}}{t}\int_{u\ge
      \frac{1}{\sqrt{2}}}u^ke^{-u}du= \frac{C(k)}{t},
    \]
    where
    $C(k)=2^\frac{k+1}{2}\bigl(2^{-\frac{k}{2}}e^{-\frac{1}{\sqrt{2}}}+k2^{-\frac{k-1}{2}}e^{-\frac{1}{\sqrt{2}}}+\dots
    +k!\bigr)$. On the arc $\vert z\vert=\frac{1}{t}$ we observe that
    $\int_{\vert z\vert=\frac{1}{t}}\vert e^{tz}\vert \vert dz\vert
    \lesssim \frac{1}{t}$. Therefore
    \[
    \int_{\vert z\vert=\frac{1}{t}}t^k\vert z\vert^k \vert e^{tz}\vert
  \vert dz\vert \lesssim \frac{1}{t}.
  \]
  By \cite[Theorem 9.6]{Rog_08} we have that
  \begin{eqnarray*}
  \vert G^{(z)}(x,y)\vert &\lesssim& C_1 \left(
    \frac{1}{t}+1\right)^{-\frac{1}{d+1}}\exp\left(-C_2
    \left(\frac{R(x,y)^{d+1}}{t}\right)^\gamma\right)\\
  &\simeq& t^{\frac{1}{d+1}}\exp\left(-C_2
    \left(\frac{R(x,y)^{d+1}}{t}\right)^\gamma\right)
\end{eqnarray*}
for $t$ small and $C_1,C_2>0$ constants independent on $x$ and
$y$. Thus, for $t$ small we have that
\begin{eqnarray*}
 \left\vert t^k\left(\frac{\partial}{\partial t}\right)^k
   h_t(x,y)\right\vert &\lesssim& \frac{1}{t} t^{\frac{1}{d+1}}\exp\left(-C_2
   \left(\frac{R(x,y)^{d+1}}{t}\right)^\gamma\right)\\
 &=&t^{-\frac{d}{d+1}}\exp\left(-C_2
   \left(\frac{R(x,y)^{d+1}}{t}\right)^\gamma\right).
\end{eqnarray*}
\end{proof}
We also need a bound on $h_t(x,y)$ for large $t$. It is well known
that, as $t\to \infty$, $h_t(x,y)$ converges to $0$ if $\Delta$ has
Dirichlet boundary condition and to the constant $\frac{1}{\mu(X)}$,
the square of the constant eigenfunction with $0$ eigenvalue, in the
Neumann case.
\begin{thm}\label{th:est2}
  Let $X$ be an affine nested fractal or an infinite blow-up of such a
  fractal. If $X$ is compact and $\Delta$ has Dirichlet boundary
  condition or $X$ is non-compact and $\Delta$ has Neumann boundary condition then
  \begin{equation}
    \label{eq:heat_kern_large_t_Dir}
    \vert h_t(x,y)\vert \lesssim t^{-\frac{d}{d+1}}, \text{ for }
    t\in[1,\infty).    
  \end{equation}
  If $X$ is compact and $\Delta$ has Neumann boundary condition then
  \begin{equation}
    \label{eq:hear_kern_lart_t_Neu}
    \left\vert h_t(x,y)-\frac{1}{\mu(X)}\right\vert \lesssim t^{-\frac{d}{d+1}}, \text{ for }
    t\in[1,\infty).    
  \end{equation}
  Moreover, similar estimates are true for
  $t^k\frac{\partial^k}{\partial t^k}h_t(x,y)$.
\end{thm}
\begin{proof}
  Assume that $X$ is compact. An easy argument may be made from Theorem 9.2 of \cite{Rog_08} which
  implies that if $\lambda_1$ is the smallest positive eigenvalue of
  $-\Delta$ then
  \[   
    \left\vert \sum_{2^n\le \frac{\lambda_j}{\lambda_1}\le 2^{n+1}}
      e^{-\lambda_jt}\varphi_j(x)\varphi_j(y)\right\vert\lesssim e^{-2^nt}\bigl(2^{n+1}\bigr)^\frac{d}{d+1}.
    \] Thus
    \begin{equation}
      \label{eq:sum_eigenv}
      \left\vert \sum_{j=1}^\infty
        e^{-\lambda_jt}\varphi_j(x)\varphi_j(y)\right\vert \lesssim
      \sum_{n=0}^\infty e^{-\lambda_12^nt}(\lambda_12)^{(n+1)\frac{d}{d+1}}.
    \end{equation}
    One can bound the right-hand side by a constant depending on
    $\lambda_1$ multiplied by
    $\int_0^\infty e^{-ut}u^{\frac{d}{d+1}}\frac{du}{u}$ which, in
    turn, equals $t^{-\frac{d}{d+1}}\Gamma\left(\frac{d}{d+1}\right)$,
    where $\Gamma$ is the gamma function. In the Dirichlet case, the
    left hand side of \eqref{eq:sum_eigenv} is $\vert h_t(x,y)\vert$,
    while in the compact Neumann case the left hand side of
    \eqref{eq:sum_eigenv} is $\left\vert
      h_t(x,y)-\frac{1}{\mu(X)}\right\vert$.

    To estimate $\frac{\partial^k}{\partial t^k}h_t(x,y)$ one can
    repeat the argument above and obtain that
    \[
    \left\vert \sum_{j=1}^\infty
        e^{-\lambda_jt}\lambda_j^k\varphi_j(x)\varphi_j(y)\right\vert
      \lesssim t^{-\frac{d}{d+1}-k} .
      \]
      Finally, if $X$ is an infinite blow-up of an affine nested
      fractal one can use again \cite[Theorem 9.2]{Rog_08} to obtain
      \[
      \left\vert \sum_{\lambda_j\le
          \frac{1}{t}}e^{-\lambda_jt}\varphi_j(x)\varphi_j(y)\right\vert
      \lesssim t^{-\frac{d}{d+1}}.
      \]
\end{proof}

The following lemma implies that the heat kernel is integrable with respect
to $x$ and $y$, and will be important later. The estimate is
presumably well-known and the proof is standard.
\begin{lem}
\label{lem:int_exp_X}If $y\in X$ we have that \[
\int_{X}e^{-c\bigl(\frac{R(x,y)^{d+1}}{t}\bigr)^{\gamma}}d\mu(x)\lesssim t^{\frac{d}{d+1}}\]
for all $t>0$. Similar estimates hold if we integrate with respect
to $y$.\end{lem}

\section{Singular Integral and Calder\'on-Zygmund operators on Fractals}\label{sec:CZO}

In this section we define singular integral and Calder\'on-Zygmund
operators on fractals and infinite blow-ups of fractals without
boundary. For this we do not need to assume that $K$ is nested, but only that it is a PCF fractal supporting a Laplacian in the sense of Kigami~\cite{Kig_CUP01}.  As usual, $X$ is either $K$ or an infinite blow-up
of $K$ without boundary. The following definition can be made for any dense subspace of $L^{2}$, but we consider only the subspace $\D$ of finite linear combinations of eigenfunctions.

\begin{defn}[{\cite[Section I.6.5]{Ste_PMS43_93}}]
\label{def:SIO_CZO}
An operator $T$ bounded on $L^{2}(\mu)$ is
called \emph{a Calder\'on-Zygmund operator} if $T$ is given by integration
with respect to a kernel $K(x,y)$, that is
\[
Tu(x)=\int_{X}K(x,y)u(y)d\mu(y)\]
for $u\in\D$ and almost all $x\notin\supp u$, such that $K(x,y)$
is a function off the diagonal which satisfies the following conditions
\begin{align}
    \bigl| K(x,y)\bigr|
    &\lesssim R(x,y)^{-d}\label{eq:singint1}\\
    \bigl|K(x,y)-K(x,\overline{y})\bigr|
    &\lesssim\eta\left(\frac{R(y,\overline{y})}{R(x,\overline{y})}\right)R(x,y)^{-d},\label{eq:singint2}
    \end{align}
for some Dini modulus of continuity $\eta$ and some $c>1$. We say,
in this case, that $K(x,y)$ is \emph{a standard kernel}.

The operator $T$ is a \emph{singular integral operator}
if the kernel $K(x,y)$ is singular at $x=y$.
\end{defn}

The next theorem gives conditions which guarantee that~\eqref{eq:singint2}
holds. In the succeeding  sections we will show that the purely imaginary
Riesz and Bessel potentials satisfy the hypothesis of this theorem.
The proof is more involved in our case than the proofs of similar
results in the real case due to the lack of a mean value theorem.
\begin{thm}
\label{thm:Singular_integral}
Let $K$ be a PCF fractal with regular self-similar Dirichlet form.
Suppose that $X$ is equal to $K$ or an infinite blow-up of K without boundary.  If $\Delta_{2}K(x,y)$ is continuous
off the diagonal of $X\times X$ and
\begin{align*}
    \bigl| K(x,y)\bigr|
    &\lesssim R(x,y)^{-d}\\
    \bigl|\Delta_{2}K(x,y)\bigr|
    &\lesssim R(x,y)^{-2d-1}
    \end{align*}
then
\[
    \bigl| K(x,y)-K(x,\overline{y})\bigr|
    \lesssim\left(\frac{R(y,\overline{y})}{R(x,\overline{y})}\right)R(x,y)^{-d},
    \]
for all $x,y,\overline{y}\in X$ such that $R(x,y)\ge cR(y,\overline{y})$,
for some $c>1$ depending only on the scaling values $r_{j}$ of the Dirichlet form.
\end{thm}

Suppose that
$X$ is an infinite blow-up such that \[
X=\bigcup_{n=1}^{\infty}F_{w_{1}}^{-1}\cdots F_{w_{n}}^{-1}K,\]
where $w=(w_{n})$ is an infinite word. For $n\ge 0$ we write $\omega |n$ for
the finite word $w_1\dots w_n$ and $r_{\omega |n}:=r_{\omega_1}\cdot\dots\cdot r_{\omega_n}$. We say that a cell $C$ has \emph{size}
$R>0$ if $C$ is an $m$ cell such that $c_{1}r^{-1}_{\omega|-m-1}\leq
R\leq c_{1}r^{-1}_{\omega |-m}$ if $m<0$ and $c_1r_{w|m}\le R\le
c_1r_{\omega|m-1}$ if $m>0$,
where $c_{1}$ is the constant from the estimates in \cite[page
110]{Str_Prin06} (see also \cite[Lemma 1.6.1 a)]{Str_Prin06}) that
relates $R(x,y)$ with the size of the cell containing $x$ and $y$.
Then, for a cell of size $R$ we have that $\mu(C)\lesssim R^{d}$.
\begin{lem}
\label{lem:main_lemma_SingularIntegral}Suppose that $C$ is a cell
of size $R>0$. Assume that $f$ is a smooth function on $C$ such
that\begin{align}
    \bigl| f(x)\bigr|
    &\lesssim R^{-\beta}\label{eq:main_lemma_1}\\
    \bigl|\Delta f(x)\bigr|
    &\lesssim R^{-\beta-d-1}\label{eq:main_lemma_2}
    \end{align}
for all $x\in C$, where $\beta>0$ is a constant. The constants that
we omit in the expressions above may depend on $f$. Then, for all $y$
and $\overline{y}$ in the interior of $C$ we have
\begin{equation}
    \bigl| f(y)-f(\overline{y})\bigr|
    \lesssim\left(\frac{R(y,\overline{y})}{R}\right)R^{-\beta}.\label{eq:main_lemma_concl}
    \end{equation}
\end{lem}
\begin{proof}
  We claim that there exists a constant $C^\prime>0$ such that for any
  $f\in \dom(\Delta)$ and for any $x,y\in K$,
  \[
  \vert f(x)-f(y)\vert\le C^\prime R(x,y)\left(\sup_{z\in
      K}\vert\Delta f(z)\vert+\max_{p,q\in\partial K}\vert f(p)-f(q)\vert\right).
  \] Using \cite[Theorem A.1]{Kig_JFA03}, our claim implies that there
  exists $C^{\prime\prime}>0$ such that for any $f\in \dom(\Delta)$,
  and $\omega\in W_*$ and any $x,y\in K_\omega$,
  \[
  \vert f(x)-f(y)\vert \le
  C^{\prime\prime}\frac{R(x,y)}{r_\omega}\left(r_\omega\mu(K_\omega)\sup_{z\in
    K_\omega}\vert\Delta f(z)\vert+\max_{p,q\in\partial K_\omega}\vert f(p)-f(q)\vert\right).
  \] The last inequality implies the conclusion of the lemma.

For the proof of the claim, let $h$ be the harmonic function on $K$
with $h|_{\partial K}=f|_{\partial K}$.
Then\[
f(x)=h(x)-\int_{K}G(x,z)\Delta f(z)d\mu(z)\;\mbox{for all}\; z\in K,\]
where $G$ is the Green function. Using \cite[Corollary 4.6]{Kig_JFA03},
we can find $C_1>0$ such that
\begin{equation}
  \label{eq:harmonic_est}
\vert h(x)-h(y)\vert \le C_1R(x,y)\left(\max_{p,q\in \partial K}\vert h(p)-h(q)\vert\right).  
\end{equation}
Moreover, \cite[Theorem 4.5]{Kig_JFA03} implies that for any $z,y,z\in X$
\begin{equation}
  \label{eq:Green_funct}
\vert G(x,z)-G(y,z)\vert\le R(x,y).  
\end{equation}
Equations \eqref{eq:harmonic_est} and \eqref{eq:Green_funct} imply now
the claim.
\end{proof}

\begin{proof}[Proof of Theorem \ref{thm:Singular_integral}]
Let $r=\max_{i=1,\dots,n}r_i$ and let $c>r^{-3-k_{0}}$, where $k_{0}$ is such that $r^{k_{0}}<1/3$,
and fix $x,y,\overline{y}\in X$ such that $R(x,y)\ge cR(y,\overline{y})$.
Then $R(x,y)\sim R(x,\overline{y})$. Let $\{C_{n}\}$ be a partition
of cells of $X$ such that each $C_{n}$ is a cell of size $r^{k_{0}+1}R(x,\overline{y})$,
or, equivalently, of size $r^{k_{0}+1}R(x,y)$. Then there is a cell
$C$ of order some $m$ in this family that contains both $y$ and
$\overline{y}$. We
claim that $x$ and $y$, and $x$ and $\overline{y}$, respectively,
do not belong to the same or adjacent $m-1$ cells. To see this, assume that $m<0$,
the proof for $m>0$ being similar. Suppose
that $x$ and $y$ belong to the same or adjacent $m-1$ cells. By
the estimates on \cite[page 110]{Str_Prin06} (see also \cite[Lemma 1.6.1]{Str_Prin06}  \cite[Theorem 2.1]{Str_08Quant})
we have that $R(x,y)\leq c_{1}r_{\omega|-m-1}^{-1}\le r^{k_{0}}R(x,y)<R(x,y)/3$,
which is a contradiction.

Let $f_{x}(z)=K(x,z)$ for all $z\in C$. By the hypotheses, $f_{x}(\cdot)$
has continuous Laplacian on $C$ and satisfies~\eqref{eq:main_lemma_1}
and~\eqref{eq:main_lemma_2} with $R=R(x,y)$ and $\beta=d$, so
Lemma \ref{lem:main_lemma_SingularIntegral} implies the conclusion.
\end{proof}

\section{Purely imaginary Riesz potentials\label{sec:Imaginary-Powers}}

Let $X$ be a nested fractal $K$ or an infinite blow-up based on this
fractal. To simplify the notation, in the remaining of the paper we will write $h_t(x,y)$ for
the heat kernel in the case that $X$ is compact and $\Delta$ has
Dirichlet boundary condition or $X$ is non-compact, and we will write
$h_t(x,y)$ for the difference between the heat kernel and $1/\mu(X)$
if $X$ is compact and $\Delta$ has Neumann boundary condition. This
allows us to use the estimates that we established in Theorems
\ref{th:est1} and \ref{th:est2}.

We define the class of operators $(-\Delta)^{i\alpha}$, with
$\alpha\in\mathbb{R}\setminus \{0\}$.  Recall that for $\lambda>0$ and
$\alpha\in\mathbb{R}$ we have
\begin{equation}
    \lambda^{i\alpha}=C_{\alpha}\lambda\int_{0}^{\infty}e^{-\lambda t}t^{-i\alpha}dt,\label{eq:imag_power_lambda}
    \end{equation}
where $C_{\alpha}=1/\Gamma(1-i\alpha)$ and $\Gamma(z)=\int_{0}^{\infty}t^{z-1}e^{-t}dt$ if $\re z>0$ is the Gamma function.
\begin{defn}
  Let $\alpha\in\mathbb{R}$, $\alpha\ne 0$. Recall that $\D$ is the set of finite linear combinations of eigenfunctions. For $\varphi$ an eigenfunction with eigenvalue $\lambda$ we define
\[
(-\Delta)^{i\alpha} \varphi= \lambda^{i \alpha} \varphi
\] and thus, for $u\in\D$, \[
(-\Delta)^{i\alpha}u=C_{\alpha}(-\Delta)\left(\int_{0}^{\infty}e^{t\Delta}ut^{-i\alpha}dt\right),\]
where $C_{\alpha}$ is the constant from \eqref{eq:imag_power_lambda}.
\end{defn}
We show that these operators are Calder\'on-Zygmund
operators by proving that their kernels satisfy estimates of the form
\eqref{eq:singint1} and \eqref{eq:singint2}.

Before doing this we need the following lemma which says that in order to show the kernel coincides with a smooth function off the diagonal in the sense of Theorem \ref{thm:Singular_integral}, it suffices to differentiate inside the integral.
\begin{lem}\label{lem: diff inside}
If $m\in L^{ \infty}([0,\infty))$ and $u$ is a smooth function with compact support on $X$ not intersecting $\{x\}$ then
\begin{equation*}
    \int_{X} \Delta^{k}u(y) \int_{0}^{\infty} m(t) h_{t}(x,y) \, dt \, d\mu(y)
    = \int_{X} u(y) \int_{0}^{\infty}  m(t) \Delta_{2}^{k} h_{t}(x,y)\, dt.
    \end{equation*}
\end{lem}
\begin{proof}

Using the Green-Gauss formula (see, for example, \cite[Theorem 2.4.1]{Str_Prin06}) we have that
\begin{align*}
\int_{X} \Delta^{k}u(y)\int_{0}^{\infty} m(t) h_{t}(x,y)  \, dt\, d\mu(y) &= \int_{X} \int_{0}^{\infty} \Delta_2^{k}u(y) m(t) h_{t}(x,y)  \, dt\, d\mu(y)\\
&= \int_{X}\int_{0}^{\infty} u(y) m(t) \Delta_2^{k} h_{t} (x,y)\, dt\, d\mu \\
&= \int_{X} u(y)\int_{0}^{\infty} m(t)\Delta_2^{k} h_t(x,y)\, dt\,d\mu. \qedhere
\end{align*}

\end{proof}

\begin{prop}
\label{pro:Kernel_imaginary}For $\alpha\in\mathbb{R}$, $\alpha\ne 0$, define \[
K_{i\alpha}(x,y)=C_{\alpha}\int_{0}^{\infty}(-\Delta_{1})h_{t}(x,y)t^{-i\alpha}dt.\]
Then $K_{i\alpha}$ is the kernel of $(-\Delta)^{i\alpha}$, in the
sense that\begin{equation}
(-\Delta)^{i\alpha}u(x)=\int_{X}K_{i\alpha}(x,y)u(y)d\mu(y)\label{eq:integral_ker}\end{equation}
for all $u\in\D$ such that $x\notin\supp u$. Moreover, the kernel
$K_{i\alpha}(x,y)$ is smooth off the diagonal and satisfies the following
estimates
\begin{align}
    \vert K_{i\alpha}(x,y)\vert
    &\lesssim R(x,y)^{-d}\label{eq:estimate1-1}\\
    \vert\Delta_{2}K_{i\alpha}(x,y)\vert
    &\lesssim R(x,y)^{-2d-1}.\label{eq:estimate2-1}
\end{align}
\end{prop}
\begin{proof}
The proof of~\eqref{eq:integral_ker} is clear because $h_{t}(x,y)$
is the kernel of the heat operator and $u$ is a linear combination of eigenfunctions. Both the smoothness and the desired estimates
rely on the following computation using the estimate~\eqref{eq:heat_estimates} and with $l=j+k$
\begin{align*}
    \biggl| \int_{0}^{\infty}(-\Delta_{2})^{k}(-\Delta_{1})^{j+1}h_{t}(x,y)t^{-i\alpha}dt\biggr|
    &= \biggl| \int_{0}^{\infty}\frac{\partial^{l+1}}{\partial t^{l+1}}h_{t}(x,y)t^{-i\alpha}dt \biggr|\\
    & \lesssim \int_{0}^{\infty}t^{-\frac{d}{d+1}-l}e^{-c\left(\frac{R(x,y)^{d+1}}{t}\right)^{\gamma}}\frac{dt}{t}\\
    &= R(x,y)^{-d-l(d+1)}\int_{0}^{\infty}u^{\frac{d}{d+1}+l}e^{-cu^{\gamma}}\frac{du}{u}\\
    &\lesssim C(j+k) R(x,y)^{-d-(j+k)(d+1)},
    \end{align*}
where $C(m)$ denotes a constant depending only on $m$.  Since the functions in the integrand are continuous on $X\times X$ and the integral converges uniformly on compact sets away from $R(x,y)=0$ we conclude that $(-\Delta_{2})^{k} (-\Delta_{1})^{j}K_{i\alpha}(x,y)$ is continuous off the diagonal for each $j,k\geq0$.
\end{proof}
\begin{cor}
The operators $(-\Delta)^{i\alpha}$, $\alpha\in\mathbb{R}\setminus \{0\}$, are Calder\'on-Zygmund operators.\end{cor}
\begin{proof}
Observe that $(-\Delta)^{i\alpha}$ extends from $\D$ to $L^{2}(\mu)$
by the spectral theorem. By Proposition \ref{pro:Kernel_imaginary},
$(-\Delta)^{i\alpha}$ is given by integration against a kernel
$K_{i\alpha}$ that is smooth off the diagonal and satisfies estimates
\eqref{eq:estimate1-1} and \eqref{eq:estimate2-1}. Theorem \ref{thm:Singular_integral}
implies that $(-\Delta)^{i\alpha}$ is a Calder\'on-Zygmund operator.
\end{proof}
We believe that the Riesz potentials are singular integral operators,
that is, the kernel $K_{i\alpha}(x,y)$ is singular on the diagonal
for all $\alpha\in\mathbb{R}\setminus\{0\}$,  but have not succeeded in proving this.
\begin{thm}
\label{thm:Lp_bounded_imag}For  $\alpha\in\mathbb{R}\setminus \{0\}$, the operator
$(-\Delta)^{i\alpha}$ defined originally on $\D$ extends
to a bounded operator on $L^{p}(\mu)$ for $1<p<\infty$, and satisfies
weak $1$-$1$ estimates.\end{thm}
\begin{proof}
Theorem 3 of \cite[page 19]{Ste_PMS43_93} implies that $(-\Delta)^{i\alpha}$
extends to a bounded operator on $L^{p}(\mu)$ for all $1<p\le2$ and
satisfies weak $1$-$1$ estimates. A duality argument (see the proof
of Theorem 1 from \cite[page 29]{Ste_PMS30_70}) implies that $(-\Delta)^{i\alpha}$
extends to a bounded operator on $L^{p}(\mu)$ for all $2<p<\infty$.
Thus $(-\Delta)^{i\alpha}$ extends to a bounded operator on $L^{p}(\mu)$,
for all $1<p<\infty$, and satisfies weak $1$-$1$ estimates.\end{proof}
\begin{rem}
The boundedness of $(-\Delta)^{i\alpha}$ on $L^{p}(\mu)$ for $1<p<\infty$
can also be obtained using the general spectral multiplier theorem
of \cite{DuOuSi_JFA02} (see \cite[Proposition 3.2]{Str_JFA03}).
\end{rem}

\begin{rem}[Laplace type transforms]
 The only property of the function $t\mapsto t^{-i\alpha}$
used in the proof of Proposition \ref{pro:Kernel_imaginary} was its
uniform boundedness. Therefore all the above results remain valid
for a more general class of operators, namely the operators of Laplace transform type.
Recall that a function $p:[0,\infty)\to\mathbb{R}$
is said to be of \emph{Laplace transform type} if\[
p(\lambda)=\lambda\int_{0}^{\infty}m(t)e^{-t\lambda}dt,\]
where $m$ is uniformly bounded.
\begin{cor} Let $p$ be of Laplace transform type. Then we can define an
operator\[
p(-\Delta)u=(-\Delta)\int_{0}^{\infty}m(t)e^{t\Delta}udt\]
for $u\in\D$ with a kernel \begin{equation}\label{laplace_kernel}
K_{p}(x,y)=\int_{0}^{\infty}(-\Delta_{1})h_{t}(x,y)m(t)dt.
\end{equation}
The kernel $K_{p}$ is smooth off the diagonal and it satisfies
the estimates
\begin{align}
    \vert K_{p}(x,y)\vert
    &\lesssim R(x,y)^{-d}\\
    \vert\Delta_{2}K_{p}(x,y)\vert
    &\lesssim R(x,y)^{-2d-1}.
    \end{align}
\end{cor}
Therefore, for a function $p$ of Laplace transform type the operator  $p(-\Delta)$ is a Calder\'on-Zygmund operator and it extends to a
bounded operator on $L^{q}(\mu)$, $1<q<\infty$.
\end{rem}
We end this section by describing the dependence of the kernel $K_{i\alpha}$ on $\alpha$.
\begin{prop}
\label{pro:Riesz_differentiable}If $x\ne y$, the map $\alpha\mapsto K_{i\alpha}(x,y)$
is differentiable.
\end{prop}
\begin{proof}
Let $x,y\in X$ such that $x\ne y$ and $f(t,\alpha):=\Delta_{1}h_{t}(x,y)t^{-i\alpha}$.
We know that $f(\cdot,\alpha)\in L^{1}(0,\infty)$ for all
$\alpha\in\mathbb{R}$. Since\[
\frac{\partial}{\partial\alpha}f(t,\alpha)=\Delta_{1}h_{t}(x,y)(-i)t^{-i\alpha}\ln t,\]
it suffices to show that $g(t)=t^{-d/(d+1)-1}\exp\bigl(-c\left(\frac{R(x,y)^{d+1}}{t}\right)^{\gamma}\bigr)\vert\ln t\vert$
is integrable and apply a standard theorem (for example, \cite[Theorem 2.27]{Fol_99}).

As $g(t)$ is continuous on $[0,1]$ we look only at the integral over $[1,\infty)$.  Using that $\ln t\le\delta^{-1}t^{\delta}$
for any $\delta>0$ we have
\begin{equation*}
g(t)\leq \delta^{-1} t^{\delta-\frac{d}{(d+1)}-1}\exp\left(-c\left(\frac{R(x,y)^{d+1}}{t}\right)^{\gamma}\right)
\end{equation*}
which is integrable on $[1,\infty)$ provided $\delta<\frac{d}{d+1}$.
\end{proof}

\section{Bessel Potentials\label{sec:Bessel-Potentials}}

We next study the Bessel potentials on $X$, where $X$ is a nested fractal $K$ or an infinite blowup, without boundary, of $K$. Our analysis follows,
in large, \cite[Chapter 5.3]{Ste_PMS30_70} (see also \cite{Str_JFA03}).
In this spirit we consider the strictly positive operator $A=1-\Delta$.
Then $u$ is an eigenfunction of $A$ if and only if it is an eigenfunction
of $-\Delta$ and $A\varphi_{n}=(1+\lambda_{n})\varphi_{n}$. Recall
that $\D$ is the set of finite linear combinations of the eigenfunctions
$\varphi_{n}$.

To define the Bessel potentials on $X$ we recall that, for $\lambda>0$
and $\alpha\in\mathbb{C}$ with $\re\alpha<0$, we have that\begin{equation}
\lambda^{\alpha}=\frac{1}{\Gamma(-\alpha)}\int_{0}^{\infty}e^{-\lambda t}t^{-\alpha-1}dt,\label{eq:power_lambda}\end{equation}
where $\Gamma$ is the Gamma function $\Gamma(z)=\int_{0}^{\infty}t^{z-1}e^{-t}dt$
if $\re z>0$.
\begin{defn}[Bessel Potentials]
Let $\alpha\in\mathbb{C}$ with $\re\alpha<0$. For an eigenfunction $\varphi$ with eigenvalue $\lambda$ we want that $(I-\Delta)^{\alpha}\varphi=A^\alpha\varphi =(1+\lambda)^\alpha \varphi$.  This motivates us to define, for $u\in\D$,
\[
(I-\Delta)^{\alpha}u=A^{\alpha}u=\frac{1}{\Gamma(-\alpha)}\int_{0}^{\infty}t^{-\alpha-1}e^{-t}e^{t\Delta}udt.\]
\end{defn}
\begin{prop}
\label{pro:group}Let $\alpha\in\mathbb{C}$ such that $\re\alpha<0$.
\begin{enumerate}
\item If $u=\sum a_{k}\varphi_{k}\in\D$ then\[
A^{\alpha}u=\sum a_{k}(1+\lambda_{k})^{\alpha}\varphi_{k}.\]

\item If $\beta\in\mathbb{C}$ such that $\re\beta<0$
then $A^{\alpha}A^{\beta}u=A^{\alpha+\beta}u$ for all $u\in\D$.
\end{enumerate}
\end{prop}
\begin{proof}
For the first assertion, recall that by the spectral theorem, if $u=\sum a_{k}\varphi_{k}\in\D$
is a finite sum then \[
e^{t\Delta}\bigl(\sum a_{k}\varphi_{k}\bigr)=\sum a_{k}e^{-t\lambda_{k}}\varphi_{k}.\]
Therefore we can exchange the sum and the integral in\begin{eqnarray*}
A^{\alpha}\bigl(\sum a_{k}\varphi_{k}\bigr) & = & \sum a_{k}\frac{1}{\Gamma(-\alpha)}\int_{0}^{\infty}t^{-\alpha-1}e^{-(1+\lambda_{k})t}dt\varphi_{k}\\
 & = & \sum a_{k}(1+\lambda_{k})^{\alpha}\varphi_{k}.\end{eqnarray*}
The second assertion is an immediate consequence of the first.
\end{proof}

Based on the above proposition we can extend the definition of $A{}^{\alpha}$
on $\D$ to arbitrary $\alpha\in\mathbb{C}$ via $A^{\alpha}=A{}^{k}A^{\alpha-k}$,
where $k$ is an integer such that $-1\le\re\alpha-k<0$, if $\re\alpha\ge0$.
Then $\{A{}^{\alpha}\}_{\alpha\in\mathbb{C}}$ is a group so that
$A^{1}=A$.

We show next that the operators $A^{\alpha}$, $\re\alpha<0$, defined
originally on $\D$, extend to bounded operators on $L^{p}(\mu)$
for all $1\le p\le\infty$. We accomplish this by studying the kernels
of the operators. The main tools we use are the heat kernel estimates
together with Lemma \ref{lem:int_exp_X}.  Since estimates of this type will be needed several times, we give the argument for the most general kernel we will encounter.
\begin{prop}\label{pro:kernel_gen_powers}
  For $s\in\mathbb{R}$, $m\in L^\infty ([0,\infty))$, and $x\ne y$ define
  \begin{equation}
    \label{eq:kernel_gen_powers}
    L_{s,m}(x,y)=\int_0^\infty m(t)t^{\frac{s}{d+1}}h_t(x,y)e^{-t}\frac{dt}{t}.
  \end{equation}
Then
\begin{equation} \label{eq:Lsm bounds}
    \vert L_{s,m}(x,y)\vert\lesssim
    \begin{cases}
        \frac{\Gamma(s)}{d-s} \Bigl( R(x,y)^{s-d} + R(x,y)^{\frac{\gamma(s-d)}{\gamma+1}} \Bigr) e^{-R(x,y)^{\frac{\gamma(d+1)}{\gamma+1}}} &\text{ if $s<d$,}\\
        \bigl(1-\log R(x,y)\bigr) e^{-R(x,y)^{\frac{\gamma(d+1)}{\gamma+1}}} &\text{ if $s=d$,} \\
        \frac{1}{s-d} R(x,y)^{\frac{\gamma(s-d)}{\gamma+1}} e^{-R(x,y)^{\frac{\gamma(d+1)}{\gamma+1}}} & \text{ if $s> d$.}
        \end{cases}
\end{equation}
In particular, for $s>0$, $L_{s,m}(x,\cdot)\in L^1(\mu)$ for all $x\in X$ and $\Vert L_{s,m}(x,\cdot)\Vert_1\le C$, with $C$ a constant independent of $x$.  Similarly $s>\frac{d}{2}$ implies $L_{s,m}(x,\cdot)\in L^2(\mu)$ for all $x\in X$ with a uniform bound on the $L^{2}$ norm.  Moreover if $s>d$ then $L_{s,m}(x,y)$ is uniformly bounded.
\end{prop}

\begin{proof}
For $s\in \mathbb{R}$ make the substitution
$t=uR(x,y)^{d+1}$, from which
\begin{align*}
    \vert L_{s,m}(x.y)\vert
    &\lesssim \int_0^\infty t^{\frac{s-d}{d+1}-1}e^{-t}e^{- c\bigl(\frac{R(x,y)^{d+1}}{t}\bigr)^{\gamma}}\,dt\\
    &= R(x,y)^{s-d}\int_0^\infty t^{\frac{s-d}{d+1}-1}e^{-tR(x,y)^{d+1}}e^{- ct^{-\gamma}}\,dt \\
    &= R(x,y)^{s-d}\int_0^\delta t^{\frac{s-d}{d+1}-1}e^{-tR(x,y)^{d+1}}e^{- ct^{-\gamma}}\,dt \\
    &\quad + R(x,y)^{s-d}\int_\delta^\infty t^{\frac{s-d}{d+1}-1}e^{-tR(x,y)^{d+1}}e^{-ct^{-\gamma}}\,dt\\
    &=: R(x,y)^{s-d}(I_1+I_2),
\end{align*}
where $\delta=R(x,y)^{-\frac{d+1}{\gamma+1}}$.  The intervals of validity of the estimates~\eqref{eq:Lsm bounds} arise naturally in estimating $I_1$ and $I_2$.

To bound $I_1$ we use that $e^{-tR(x,y)^{d+1}}\le 1$ on the interval, and make the change of variable
$t=u^{-\gamma}$ to find
\begin{equation}\label{eq: kernel_gen_powers intermed step}
  I_1
  \le \int_{\delta^{-\gamma}}^\infty u^{-\frac{s-d}{\gamma(d+1)}}e^{-cu}\,\frac{du}{u}.
\end{equation}
The exponential decay in the integrand implies this is bounded by a constant multiple $c(s)$ of the integral over the unit length interval $[\delta^{-\gamma},\delta^{-\gamma}+1]$.  It is easy to see $c(s)\leq 1$ for $s\geq d$ and $c(s)\leq\Gamma(s)$ otherwise.  We bound the exponential term by $e^{-c\delta^{-\gamma}}$, and integrate the polynomial term to obtain $\bigl.\frac{\gamma(d+1)}{d-s} u^{-\frac{s-d}{\gamma(d+1)}}\bigr|_{\delta^{-\gamma}}^{1+\delta^{-\gamma}}$ unless $s=d$ where it is $\bigl.\log u\bigr|_{\delta^{-\gamma}}^{1+\delta^{-\gamma}}$.  If $s<d$ we bound by the value at the upper endpoint, obtaining $(1+\delta^{-\gamma})^{\frac{d-s}{d+1}}\leq 1+R(x,y)^{\frac{d-s}{\gamma+1}}$.  If $s=d$ it is easy to see the bound is by $1+\gamma\log \delta\lesssim 1-\log R(x,y)$.  And if $s>d$ we bound by the value at the lower endpoint, which is $\delta^{-\gamma\frac{d-s}{d+1}}=R(x,y)^{\frac{d-s}{\gamma+1}}$.
Combining these we have found
\begin{equation*}
    I_{1}
    \lesssim
    \begin{cases}
        \frac{1}{d-s} \bigl(1+ R(x,y)^{-\frac{s-d}{\gamma+1}}\bigr) e^{-c R(x,y)^{\frac{\gamma(d+1)}{\gamma+1}}} &\text{ if $s< d$,}\\
        \bigl( 1 - \log R(x,y) \bigr) e^{-c R(x,y)^{\frac{\gamma(d+1)}{\gamma+1}}} &\text{ if $s=d$,}\\
        \frac{1}{s-d} R(x,y)^{-\frac{s-d}{\gamma+1}} e^{-c R(x,y)^{\frac{\gamma(d+1)}{\gamma+1}}} &\text{ if $s>d$.}
        \end{cases}
    \end{equation*}

For the estimate of $I_2$ we use that $e^{-ct^{-\gamma}}\leq1$ on the interval, so that with $u=tR(x,y)^{d+1}$ we obtain
\begin{equation*}
    I_{2}
    \leq \int_\delta^\infty t^{\frac{s-d}{d+1}}e^{-tR(x,y)^{d+1}}\frac{dt}{t}
    = R(x,y)^{d-s} \int_{\delta R(x,y)^{d+1}}^\infty u^{\frac{s-d}{d+1}}e^{-u}\frac{du}{u}
    \end{equation*}
This integral is the same as in~\eqref{eq: kernel_gen_powers intermed step}, except that the power in the integrand is $\frac{s-d}{d+1}$ not $\frac{s-d}{\gamma(d+1)}$. Notice that the lower endpoint is $\delta R(x,y)^{d+1}=R(x,y)^{\frac{\gamma(d+1)}{\gamma+1}}=\delta^{-\gamma}$.  We must therefore have the same estimates for the integral that we did for $I_1$, but with the power $R(x,y)^{-\frac{s-d}{\gamma+1}}$ replaced by $R(x,y)^{-\frac{\gamma(s-d)}{\gamma+1}}$ throughout.  Multiplying through by the leading $R(x,y)^{d-s}$ factor makes these estimates the same or smaller than the corresponding ones for $I_{1}$, which completes the proof of~\eqref{eq:Lsm bounds}.

Observe that the upper bound for $s>d$ is itself uniformly bounded, so $L_{s,m}(x,y)$ is uniformly bounded in this case.  Also the bound is integrable for large $R(x,y)$ because it has exponential decay, and the the singularity in the bound (which occurs when $s\leq d$) is integrable provided $s>0$. Hence $\Vert L_{s,m}(x,\cdot)\Vert_1\le C$ with $C$ independent of $x$.  Finally we note that the singularity is in $L^{2}$ if $s=d$, or if $s<d$ and $2(s-d)+d<0$, meaning $s>\frac{d}{2}$.
\end{proof}

\begin{cor}
\label{pro:kernel_integrable}For $\re\alpha<0$ define
 \[
K_{\alpha}(x,y)=\frac{1}{\Gamma(-\alpha)}\int_{0}^{\infty}h_{t}(x,y)t^{-\alpha-1}e^{-t}dt.\]
Then $K_{\alpha}(\cdot,y)$ is integrable for all $y\in X$ with $\Vert K_{\alpha}(\cdot,y)\Vert_{1}\le C$
with $C$ independent of $y$. The same statements are true for
$K_{\alpha}(x,\cdot)$.
\end{cor}
\begin{proof}
Let $\alpha=a+ib$, with $a<0$. This is an immediate consequence of the $L^{1}$ estimate in Proposition~\ref{pro:kernel_gen_powers} with $s=-a(d+1)$ and $m(t)=t^{ib}$.
\end{proof}

As a consequence of the above result we obtain that the operators
$A^{\alpha}$ are bounded operators on $L^{p}(\mu)$ for $1\le p\le\infty$
if $\re\alpha<0$. We prove this statement in the following.
\begin{thm}
\label{thm:Lp_negative}Let $\alpha$ be such that $\re\alpha<0$.
Then $A^{\alpha}$ is given by integration with respect to $K_{\alpha}$,
that is\begin{equation}
A^{\alpha}f(x)=\int_{X}K_{\alpha}(x,y)f(y)d\mu(y),\label{eq:kernel_Bessel}\end{equation}
for all $f\in\D$. In particular, the operator $A^{\alpha}$ defined
originally on $\D$ extends to a bounded operator on $L^{p}(\mu)$
for all $1\le p\le\infty$.\end{thm}
\begin{proof}
The proof of \eqref{eq:kernel_Bessel} is clear. The second part follows
immediately from the estimates of Corollary \ref{pro:kernel_integrable}
by an argument analogous to the classical proof of the Young's inequality via the generalized Minkowski
inequality.
\end{proof}

\begin{rem*}
The boundedness of the operators $A^{\alpha}$ on $L^{p}(\mu)$, for
$1<p<\infty$, can also be obtained using the spectral theorems of
\cite{DuOuSi_JFA02} (see \cite[Proposition 3.2]{Str_JFA03}).\end{rem*}

\begin{prop}
\label{pro:Hilbert-Schmidt}On the compact set $K$ the operator $A^{\alpha}$
is Hilbert-Schmidt when $\re\alpha<-\frac{d}{2(d+1)}$.
\end{prop}
\begin{proof}
Let $\alpha=a+ib$, set $s=-a(d+1)$ and $m(t)=t^{ib}$.  Then $s>\frac{d}{2}$ so we can apply Proposition~\ref{pro:kernel_gen_powers} to see $\|K_{\alpha}(x,\cdot)\|_{L^{2}}$ is uniformly bounded.  Integrating with respect to $x$ produces a factor of $\mu(K)<\infty$, so $K_{\alpha}(x,y)$ is in $L^{2}$ of the product space.
\end{proof}
\begin{cor}\label{cor: Hilbert Schmidt}
Assume that $X$ is compact and $\re\alpha<-\frac{d}{2(d+1)}$.
Then\[
K_{\alpha}(x,y)=\sum_{n}(1+\lambda_{n})^{\alpha}\varphi_{n}(x)\varphi_{n}(y),\]
where the infinite sum converges in $L^{2}(\mu\times\mu)$.
\end{cor}
\begin{prop}\label{prop Bessel potential dependence on alpha}
If $\re \alpha<0$ then $K_{\alpha}$ is smooth off the diagonal. If in addition $\re \alpha\leq - \frac{d}{d+1}$ then
$K_{\alpha}$ is continuous and uniformly bounded. Also, the map $\alpha\mapsto K_{\alpha}(x,y)$ is analytic on $\{\re\alpha<0\}$, for all $x,y\in X$ with $x\ne y$.
\end{prop}
\begin{proof}
Recall
$K_{\alpha}(x,y)=\int_{0}^{\infty}h_{t}(x,y)t^{-\alpha-1}e^{-t}dt$
with $\alpha=a+ib$.  To obtain smoothness off the diagonal, it
suffices by Lemma~\ref{lem: diff inside} that we differentiate inside
the integral.  Since $h_{t}$ is the heat kernel, applying the
Laplacian is the same as differentiating with respect to $t$, and we
know $t^{k}\frac{\partial^{k}}{\partial t^{k}}h_{t}$ satisfies the
same bounds as $h_t$ itself.  Then applying
Proposition~\ref{pro:kernel_gen_powers} for $m(t)=t^{ib}$ and
$s=-(a+j+k)(d+1)$ one can see that $\Delta^k_1\Delta^j_2
K_\alpha(x,y)$ is continuous off the diagonal for all $j,k\ge 1$. If
$\re \alpha\leq -\frac{d}{d+1}$ then a second application of the
Proposition shows $K_\alpha(x,y)$ is uniformly bounded. 

The second part follows by a standard argument (such as~\cite[Theorem 5.4]{StSha_03}), since
the map $F(\alpha,t)=h_{t}(x,y)t^{-\alpha-1}e^{-t}$ is analytic on
$\alpha$ for each $t>0$, and continuous in $\alpha$ and $t$.
\end{proof}

\subsection{Purely imaginary Bessel potentials}

We turn our attention now to the study of the kernel of purely imaginary
Bessel potentials, that is, operators of the form $(I-\Delta)^{i\alpha}$,
with $\alpha\in\mathbb{R}\setminus \{0\}$. We use formula \eqref{eq:imag_power_lambda}
as the starting point and, for $\alpha\in\mathbb{R}\setminus \{0\}$ and $u\in\D$,
we define\[
(I-\Delta)^{i\alpha}u=C_{\alpha}(I-\Delta)\int_{0}^{\infty}e^{t\Delta}ue^{-t}t^{i\alpha}dt.\]
For $\alpha\in\mathbb{R}\setminus \{0\}$, we define the kernel of $(I-\Delta)^{i\alpha}$
via
\begin{equation}
G_{i\alpha}(x,y)=i\alpha
C_{\alpha}\int_{0}^{\infty}h_{t}(x,y)e^{-t}t^{i\alpha-1}dt.\label{eq:ker_imag_Bessel}
\end{equation}

\begin{thm}
\label{thm:Kernel_Bessel_Imaginary}For $\alpha\in\mathbb{R}\setminus \{0\}$, $G_{i\alpha}(x,y)$
defined in \eqref{eq:ker_imag_Bessel} is the kernel of $(I-\Delta)^{i\alpha}$,
in the sense that\begin{equation}
(I-\Delta)^{i\alpha}u(x)=\int_{X}G_{i\alpha}(x,y)u(y)d\mu(y)\label{eq:kern_Besselop}\end{equation}
for all $u\in\D$ such that $x\notin\supp u$. Moreover, $G_{i\alpha}(x,y)$
is smooth off the diagonal.\end{thm}
\begin{proof}
 As $u\in\D$ is a finite sum, using integration by parts we have that
\begin{eqnarray*}
  (I-\Delta)^{i\alpha}u&=&C_\alpha\int_X\left( \int_0^\infty
    h_t(x,y)e^{-t}t^{i\alpha}dt+\int_0^\infty \Bigl(\frac{\partial}{\partial
      t}h_t(x,y)\Bigr)e^{-t}t^{i\alpha}dt\right)u(y)\,d\mu(y)\\
&=&i\alpha C_\alpha\int_X\int_0^\infty h_t(x,y)e^{-t}t^{i\alpha-1}\,dt\,u(y)\,d\mu(y).
\end{eqnarray*}
Thus \eqref{eq:kern_Besselop} holds. From Lemma~\ref{lem: diff inside} we may apply powers of the Laplacian inside the integral in~\eqref{eq:ker_imag_Bessel} to establish smoothness off the diagonal.  Moreover the application of $\Delta_{1}^j\Delta^k_2$ is equivalent to replacing $h_t(x,y)$ with $\frac{\partial^{(j+k)}}{\partial t^{(j+k)}}h_{t}$, which satisfies the same estimates as $t^{-(j+k)}h_{t}$. Applying Proposition
\ref{pro:kernel_gen_powers} with $s=-(j+k)(d+1)$, we see that
\[
\vert \Delta_{1}^j\Delta^k_2 G_{i\alpha}(x,y)\vert \lesssim R(x,y)^{-(j+k)(d+1)-d}e^{-R(x,y)^{\frac{\gamma(d+1)}{\gamma+1}}},
\]
for all $j,k\ge 0$, so in particular $G_{i\alpha}(x,y)$ is smooth off the diagonal.
\end{proof}

\begin{cor}
 For $\alpha\in\mathbb{R}\setminus \{0\}$, the operator $(I-\Delta)^{i\alpha}$ is a Calder\'on-Zygmund operator.
\end{cor}
\begin{proof}
  The operator $(I-\Delta)^{i\alpha}$
extends to a bounded operator on $L^{2}(\mu)$ by the spectral theorem. The
 proof of Theorem \ref{thm:Kernel_Bessel_Imaginary} implies that
$G_{i\alpha}$ satisfies the estimates \eqref{eq:intro_est1} and
\eqref{eq:intro_est2}. Thus $(I-\Delta)^{i\alpha}$ is a Calder\'on-Zygmund operator.
\end{proof}
\begin{cor}
If $\alpha\in\mathbb{R}\setminus \{0\}$, then the operator $(I-\Delta)^{i\alpha}$
defined originally on $\D$ extends to a bounded operator on $L^{p}(\mu)$
for $1<p<\infty$ and satisfies weak $1$-$1$ estimates.\end{cor}
\begin{rem}
The boundedness of $(I-\Delta)^{i\alpha}$ on $L^{p}(\mu)$ for $1<p<\infty$
can also be obtained using the general spectral multiplier theorem
of \cite{DuOuSi_JFA02} (see \cite[Proposition 3.2]{Str_JFA03}).
\end{rem}
We also note that one can easily modify the  proof
of Proposition~\ref{pro:Riesz_differentiable} to obtain the following
result.
\begin{prop}
Let $x,y\in X$ with $x\ne y$. Then the map $\alpha\mapsto G_{i\alpha}(x,y)$
is differentiable.
\end{prop}

\section{\label{sec:Products}Complex powers on products of fractals and blowups}

In this section we extend our analysis of Calder\'on-Zygmund operators
and the Riesz and Bessel potentials to finite products $X^N$, where $X$ is
either a nested fractal $K$ or an infinite blow-up without boundary of $K$. The
study of the energy, the Laplace operator, and the heat kernel estimates
on products of PCF fractals was initiated by Strichartz in~\cite{Str:TAM04}
(see, also, \cite{Str_Prin06,Str_08Quant}). We begin by reviewing
the basic steps in his construction. Consider the product space
$X^{N}$ with the product measure $\mu^{N}$.
Notice that $X^{N}$ is not, in general, a PCF fractal.
Recall from \cite{Str:TAM04} that a measurable function $u$ on $X^2$ has \emph{minimal regularity} if and
only if for almost $x_{2}\in X$, $u(\cdot,x_{2})\in\dom\E$,
and for almost every $x_{1}\in X$, $u(x_{1},\cdot)\in \dom\E$. Such a
function belongs to the domain of the energy on $X^{2}$, $\dom\E^{2}$,
if and only if
\[
\E^{2}(u)=\int_{X}\E(u(\cdot,x_{2}))d\mu(x_{2})+\int_{X}\E(u(x_{1},\cdot))d\mu(x_{1})
\]
exists and is finite \cite{Str:TAM04,Str_Prin06}. This definition can be easily
generalized to $X^N$. Then we may define
a Laplacian by the weak formulation
\[
\E^{N}(u,v)=-\int_{X^{N}}(\Delta u)v\,d\mu^{N}.
\]
To avoid confusion, we will henceforth write $\Delta^{\prime}$
for the Laplacian on $X$. Recall that we fixed orthonormal basis
$\{\varphi_{n}\}_{n}$ for $L^{2}(\mu)$ such that each $\varphi_{n}$
is an eigenfunction of $\Delta^{\prime}$. Then if
$\underline{n}=(n_1,n_2,\dots,n_N)$ the functions
\[
\varphi_{\underline{n}}(x)=\varphi_{n_1}(x_{1})\varphi_{n_2}(x_{2})\dots \varphi_{n_N}(x_{N}),
\]
where $x=(x_{1},x_{2},\dots,x_{N})\in X^{N}$, form an orthonormal basis for $L^{2}(\mu^{N})$.
Let $\D^{N}$ be the set of finite linear combinations of $\varphi_{\underline{n}}$.

The heat kernel on $X^{N}$ is the product \begin{equation}
h_{t}^{N}(x,y)=h_{t}(x_{1},y_{1})h_{t}(x_{2},y_{2})\dots h_{t}(x_N,y_N),\label{eq:heat_kernel_prod}\end{equation}
for $x,y\in X^{N}$ (\cite[Theorem 6.1]{Str_08Quant}).
We extend the metric to $X^{N}$ by\[
R^{N}(x,y)=\left(R(x_{1},y_{1})^{(d+1)\gamma}+R(x_{2},y_{2})^{(d+1)\gamma}+\dots
+R(x_{N},y_{N})^{(d+1)\gamma}\right)^{1/((d+1)\gamma)}.\]
Then, if $K$ is an affine nested fractal, the heat kernel estimates
become \cite[Theorem 2.2]{Str_08Quant}
\begin{equation}
h_{t}^{N}(x,y)\lesssim
t^{-\frac{Nd}{d+1}}\exp\left(-c\frac{R^{N}(x,y)^{d+1}}{t}\right)^{\gamma}.\label{eq:heat_estimates_prod}
\end{equation}
Theorems \ref{th:est1} and \ref{th:est1} can be easily extended to the
product setting to show that the same estimates are satisfied by $t^{k}\frac{\partial^{k}}{\partial t^{k}}h_{t}^{N}(x,y)$
for all $t>0$.

\subsection{Singular integral and Calder\'on-Zygmund operators on products}

 We say that on operator $T$ bounded on $L^{2}(\mu^{N})$
is a \emph{Calder\'on-Zygmund operator} on $X^{N}$ if it is given
by integration with respect to a kernel $K(x,y)$ which is a function
off the diagonal and satisfies
\begin{equation}
\vert K(x,y)\vert\lesssim R^{N}(x,y)^{-Nd}\label{eq:singint1_prod}\end{equation}
for all $x\ne y$ and
\begin{equation}
\vert K(x,y)-K(x,\overline{y})\vert\lesssim\eta\left(\frac{R^{N}(y,\overline{y})}{R^{N}(x,\overline{y})}\right)R^{N}(x,y)^{-Nd},\label{eq:singint2_prod}\end{equation}
if $R^{N}(x,y)\ge cR^{N}(y,\overline{y})$, for some $c>1$, where
$\eta$ is a Dini modulus of continuity. We say that $T$ is a \emph{singular integral operator} if $K(x,y)$
is singular at $x=y$. The next theorem, which is the main result
of this section, extends Theorem \ref{thm:Singular_integral} to the
product setting.
\begin{thm}
\label{thm:Singintegral_product}Let $K$ be a nested fractal and assume that $X$
is either $K$ or an infinite blow-up
of K without boundary. Suppose that $T:L^{2}(\mu^{N})\to L^{2}(\mu^{N})$ is given
by integration with respect to a kernel $K(x,y)$ which is smooth
off the diagonal and satisfies the following estimates
\begin{eqnarray}
\vert K(x,y)\vert & \lesssim & R^{N}(x,y)^{-Nd}\label{eq:singint_1}\\
\vert\Delta_{y,i}^{\prime}K(x,y)\vert & \lesssim & R^{N}(x,y)^{-(N+1)d-1},\; i=1,2,,\dots,N\label{eq:singint_2}\end{eqnarray}
where for $y=(y_{1},y_{2},\dots,y_{N})$, $\Delta_{y,i}^{\prime}K(x,y)$ is the
Laplacian on $X$ with respect to $y_{i}$. Then $T$ is a Calder\'on-Zygmund
operator. In particular, $T$ extends to a bounded operator
on $L^{p}(\mu^{N})$ for all $1<p<\infty$ and satisfies weak $1$-$1$
estimates.\end{thm}
\begin{proof}
We prove the theorem for $N=2$. The difference between this case and general $N$
is merely notation. Let $c^{\prime}$ be the constant from Theorem \ref{thm:Singular_integral}
and let $c=2c^{\prime}$. Let $x=(x_{1},x_{2}),y=(y_{1},y_{2})$,
and $\overline{y}=(\overline{y}_{1},\overline{y}_{2})$ in $X^{2}$
such that $R^{2}(x,y)\ge cR^{2}(y,\overline{y})$. Then we can find
$C=C_{1}\times C_{2}$, where $C_{1}$ and $C_{2}$ are cells of size
$r^{k_{0}+1}R^{2}(x,y)$, such that $y,\overline{y}\in C$ and $x\notin C$.
Moreover, we have that $R(x_{1},y_{1})\ge c^{\prime}R(y_{1},\overline{y}_{1})$
or $R(x_{2},y_{2})\ge c^{\prime}R(y_{2},\overline{y}_{2})$. Assume
that $R(x_{1},y_{1})\ge c^{\prime}R(y_{1},\overline{y}_{1})$. Then
\begin{align*}
    K(x,y)-K(x,\overline{y})
    &= K(x,(y_{1},y_{2}))-K(x,(y_{1},\overline{y}_{2}))\\
    &\quad + K(x,(y_{1},\overline{y}_{2}))-K(x,(\overline{y}_{1},\overline{y}_{2})).
    \end{align*}
Let $u_{(x,y_{1})}(z)=K(x,(y_{1},z))$ for $z\in C_{2}$ and let $v_{(x,\overline{y}_{2})}(z)=K(x,(z,\overline{y}_{2}))$
for $z\in C_{1}$. Since $R(x_{1},y_{1})\ge c^{\prime}R(y_{1},\overline{y}_{1})$
we have that $u_{(x,y_{1})}$ is smooth on $C_{2}$ and $v_{(x,\overline{y}_{2})}$
is smooth on $C_{1}$. Moreover, they satisfy the following estimates\begin{eqnarray*}
\vert u_{(x,y_{1})}(z)\vert & \lesssim & R^{2}(x,y)^{-2d},\\
\vert\Delta^{\prime}u_{(x,y_{1})}(z)\vert & \lesssim & R^{2}(x,y)^{-3d-1},\end{eqnarray*}
for all $z\in C_{1}$, and\begin{eqnarray*}
\vert v_{(x,\overline{y}_{2})}(z)\vert & \lesssim & R^{2}(x,y)^{-2d},\\
\vert\Delta^{\prime}v_{(x,\overline{y}_{2})}(z)\vert & \lesssim & R^{2}(x,y)^{-3d-1},\end{eqnarray*}
for all $z\in C_{2}$. Lemma \ref{lem:main_lemma_SingularIntegral}
together with the fact that $R(y_{i},\overline{y}_{i})\leq R^{2}(y,\overline{y})$,
$i=1,2$, implies that\[
\vert u_{(x,y_{1})}(y_{2})-u_{(x,y_{1})}(\overline{y}_{2})\vert\lesssim\left(\frac{R^{2}(y,\overline{y})}{R^{2}(x,y)}\right)R(x,y)^{-2d}\]
and \[
\vert v_{(x,\overline{y}_{2})}(y_{1})-v_{(x,\overline{y}_{2})}(\overline{y}_{1})\vert\lesssim\left(\frac{R^{2}(y,\overline{y})}{R^{2}(x,y)}\right)R(x,y)^{-2d}.\]
Thus \eqref{eq:singint2_prod} holds with $\eta(s)=s$.
\end{proof}

\subsection{Purely imaginary Riesz potentials on products}

Having all the ingredients in place, we can define for $\alpha\in\mathbb{R}\setminus\{0\}$
and $u\in\D^{N}$
\[
(-\Delta)^{i\alpha}u=C(-\Delta)\int_{0}^{\infty}e^{t\Delta}ut^{-i\alpha}dt.\]
The kernel of $(-\Delta)^{i\alpha}$ is given by the formula
\[
K_{i\alpha}(x,y)=C\int_{0}^{\infty}(-\Delta_{1})h_{t}^{N}(x,y)t^{-i\alpha}dt.\]
Using, basically, the same computations as in Section \ref{sec:Imaginary-Powers}
we see that, for $\alpha\in\mathbb{R}\setminus\{0\}$, $K_{i\alpha}(x,y)$ is smooth
away from the diagonal and satisfies the following estimates:
\begin{eqnarray}
\vert K_{i\alpha}(x,y)\vert & \lesssim & R^{N}(x,y)^{-Nd},\label{eq:eq1_prod}\\
\vert\Delta_{y,i}^{\prime}K_{i\alpha}(x,y)\vert & \lesssim &  R^{N}(x,y)^{-(N+1)d-1},\; i=1,2,\dots,N,\label{eq:lapl_ker2_i_est}\end{eqnarray}
so $(-\Delta)^{i\alpha}$ is a Calder\'on-Zygmund operator on $X^{N}$.
We believe that they are singular integral operators but have not succeeded
in proving this.
\begin{cor}
For $a\in\mathbb{R}\setminus\{0\}$, the operator
 \[
(-\Delta)^{i\alpha}u(x)=\int_{X^{N}}K_{i\alpha}(x,y)u(y)\,d\mu^{N}(y)\]
extends to a bounded operator on $L^{p}(\mu^{N})$, for all $1<p<\infty$,
and satisfies weak $1$-$1$-estimates.
\end{cor}
\begin{rem}
The boundedness of $(-\Delta)^{i\alpha}$ on $L^{p}(\mu^{N})$, $1<p<\infty$,
can also be deduced using the multivariable spectral results of \cite{sikora-2008}.
\end{rem}
The results about the dependence of the kernels on $\alpha$ extend
easily to the product setting. Using essentially the proof of
Proposition~\ref{pro:Riesz_differentiable} we see that
$\alpha\mapsto K_{i\alpha}(x,y)$ is differentiable in $\alpha$ for
all $x,y\in X$ with $x\ne y$.

\subsection{Bessel Potentials on products}

Consider now the strictly positive operator $A=I-\Delta$. As before,
for $\re\alpha<0$ and $u\in\D^{N}$, we define
\begin{equation}
A^{\alpha}u=\frac{1}{\Gamma(-\alpha)}\int_{0}^{\infty}e^{-t}t^{-\alpha-1}e^{t\Delta}udt.\label{eq:complex_pow_prod}
\end{equation}
with the kernel
 \[
K_{\alpha}(x,y)=\int_{0}^{\infty}h_{t}^{N}(x,y)e^{-t}t^{-\alpha-1}dt.\]
It is clear that the equivalent of~Proposition \ref{pro:group} holds so
we can define $A^{\alpha}=A^{\alpha-k}A^{k}$, where $k$ is such
that $-1\le\re\alpha-1<0$, if $\re\alpha\ge0$. Versions of all the statements
established in Section~\ref{sec:Bessel-Potentials} remain valid for
this class of operators. The crucial ingredients in the proofs there were the heat kernel estimates~\eqref{eq:heat_estimates}
which we have now as~\eqref{eq:heat_estimates_prod}.

In particular, we see from~\eqref{eq:heat_estimates_prod} that all of
the integrals we encounter in the product setting differ
from~\eqref{eq:kernel_gen_powers} only in that there is an extra
factor of $t^{-\frac{(N-1)d}{d+1}}$, so
Proposition~\ref{pro:kernel_gen_powers} is valid if we replace
$h_{t}(x,y)$ with $h^{N}_{t}(x,y)$ and replace each occurrence of
$s-d$ in~\eqref{eq:Lsm bounds} with $s-Nd$, making the regions for the
estimates $s<Nd$, $s=Nd$ and $s>Nd$. 

The above quickly gives an analogue of
Theorem~\ref{thm:Lp_negative}. Notice that the singularity occurring
for $s\leq Nd$ is in $L^{1}(d\mu^{N})$ if $s>0$, so that $\Vert
K_{\alpha}(\cdot,y)\Vert_{1}$ is bounded by a constant independent of
$y$ and similarly for $\Vert K_{\alpha}(x,\cdot)\Vert_{1}$.  This
implies that for $\re\alpha<0$, $A^{\alpha}$ extends to be a bounded
operator on $L^{p}(\mu^{N})$ for all $1\le p\le\infty$. 

We also obtain product versions of
Proposition~\ref{pro:Hilbert-Schmidt} and Corollary~\ref{cor: Hilbert
  Schmidt}.  The singularity $R(x,y)^{s-Nd}$ is in $L^{2}(d\mu^{N})$
if $2(s-Nd)+Nd>0$, so if $s>\frac{Nd}{2}$.  As $s=-\re(\alpha)(d+1)$
we conclude that $\|K_{\alpha}(x,\cdot)\|_{L^{2}}$ is uniformly
bounded for $\re(\alpha)<-\frac{Nd}{2(d+1)}$.  Thus on the compact
fractal $K$ the kernel is in $L^{2}$, the operator $A^{\alpha}$ is
Hilbert-Schmidt, and we have the $L^{2}$ expansion 
\[
K_{\alpha}(x,y)=\sum_{\underline{n}}(1+\lambda_{n_1}+\lambda_{n_2}+\dots +\lambda_{n_N})^{\alpha}\varphi_{\underline{n}}(x)\varphi_{\underline{n}}(y).\]

Our estimates show that the kernel is always smooth away from the
diagonal.  If $\re(\alpha)<-\frac{Nd}{d+1}$ then we have $s>Nd$, from
which the kernel is also globally continuous and uniformly bounded.
By the same argument as in Proposition~\ref{prop Bessel potential
  dependence on alpha}  the map $\alpha\mapsto K_{\alpha}(x,y)$ is
analytic on $\{\re\alpha<0\}$ for all  $x\neq y\in X^{N}$. 

For purely  imaginary Bessel potentials $(I-\Delta)^{i\alpha}$ we have
analogues of Theorem~\ref{thm:Kernel_Bessel_Imaginary} and its
corollaries.  Specifically, for $\alpha\in\mathbb{R}\setminus\{0\}$ we define
\[
G_{i\alpha}(x,y)=i\alpha C_{\alpha}\int_{0}^{\infty}h_{t}^{N}(x,y)\bigr)e^{-t}t^{i\alpha-1}dt.\]
and verify that it represents $(I-\Delta)^{i\alpha}$ on those $u\in
D^{N}$ with support away from $x$.  By the previous reasoning about
the analogue of Proposition~\ref{pro:kernel_gen_powers} (with $d$
replaced by $Nd$ in the conclusions) we see that $G_{i\alpha}(x,y)$ is
smooth off the diagonal and satisfies~\eqref{eq:eq1_prod}
and~\eqref{eq:lapl_ker2_i_est}. Thus 
$(I-\Delta)^{i\alpha}$ is a Calder\'on-Zygmund operator and it extends to a bounded operator on $L^{p}(\mu^N)$
for all $1<p<\infty$ and satisfies weak $1$-$1$ estimates.  The map $\alpha\mapsto G_{i\alpha}(x,y)$
is also differentiable for all $x\ne y$. 

\bibliographystyle{amsplain}
\bibliography{bibliography}

\providecommand{\bysame}{\leavevmode\hbox to3em{\hrulefill}\thinspace}
\providecommand{\MR}{\relax\ifhmode\unskip\space\fi MR }
\providecommand{\MRhref}[2]{%
  \href{http://www.ams.org/mathscinet-getitem?mr=#1}{#2}
}
\providecommand{\href}[2]{#2}
\begin{thebibliography}{10}

\bibitem{BaPe_PTRF88}
Martin~T. Barlow and Edwin~A. Perkins, \emph{Brownian motion on the
  {S}ierpi\'nski gasket}, Probab. Theory Related Fields \textbf{79} (1988),
  no.~4, 543--623. \MR{MR966175 (89g:60241)}

\bibitem{BassStrTep_JFA99}
Oren Ben-Bassat, Robert~S. Strichartz, and Alexander Teplyaev, \emph{What is
  not in the domain of the {L}aplacian on {S}ierpinski gasket type fractals},
  J. Funct. Anal. \textbf{166} (1999), no.~2, 197--217. \MR{MR1707752
  (2001e:31016)}

\bibitem{Dav_CTM90}
E.~B. Davies, \emph{Heat kernels and spectral theory}, Cambridge Tracts in
  Mathematics, vol.~92, Cambridge University Press, Cambridge, 1990.
  \MR{MR1103113 (92a:35035)}

\bibitem{FiHaKu_94}
Pat~J. Fitzsimmons, Ben~M. Hambly, and Takashi Kumagai, \emph{Transition
  density estimates for {B}rownian motion on affine nested fractals}, Comm.
  Math. Phys. \textbf{165} (1994), no.~3, 595--620. \MR{MR1301625 (95j:60122)}

\bibitem{Fol_99}
Gerald~B. Folland, \emph{Real analysis}, second ed., Pure and Applied
  Mathematics (New York), John Wiley \& Sons Inc., New York, 1999, Modern
  techniques and their applications, A Wiley-Interscience Publication.
  \MR{MR1681462 (2000c:00001)}

\bibitem{FuSh_92}
M.~Fukushima and T.~Shima, \emph{On a spectral analysis for the {S}ierpi\'nski
  gasket}, Potential Anal. \textbf{1} (1992), no.~1, 1--35. \MR{MR1245223
  (95b:31009)}

\bibitem{HaKu_PLMS99}
B.~M. Hambly and T.~Kumagai, \emph{Transition density estimates for diffusion
  processes on post critically finite self-similar fractals}, Proc. London
  Math. Soc. (3) \textbf{78} (1999), no.~2, 431--458. \MR{MR1665249
  (99m:60118)}

\bibitem{HaKu_PTRD03}
\bysame, \emph{Diffusion processes on fractal fields: heat kernel estimates and
  large deviations}, Probab. Theory Related Fields \textbf{127} (2003), no.~3,
  305--352. \MR{MR2018919 (2004k:60219)}

\bibitem{HuZa_05}
Jiaxin Hu and Martina Z{\"a}hle, \emph{Potential spaces on fractals}, Studia
  Math. \textbf{170} (2005), no.~3, 259--281. \MR{MR2185958 (2006h:31008)}

\bibitem{HuZa_09}
\bysame, \emph{Generalized {B}essel and {R}iesz potentials on metric measure
  spaces}, Potential Anal. \textbf{30} (2009), no.~4, 315--340. \MR{MR2491456}

\bibitem{Hut_81}
John~E. Hutchinson, \emph{Fractals and self-similarity}, Indiana Univ. Math. J.
  \textbf{30} (1981), no.~5, 713--747. \MR{MR625600 (82h:49026)}

\bibitem{Kig_CUP01}
Jun Kigami, \emph{Analysis on fractals}, Cambridge Tracts in Mathematics, vol.
  143, Cambridge University Press, Cambridge, 2001. \MR{MR1840042
  (2002c:28015)}

\bibitem{Kig_JFA03}
\bysame, \emph{Harmonic analysis for resistance forms}, J. Funct. Anal.
  \textbf{204} (2003), no.~2, 399--444. \MR{2017320 (2004m:31010)}

\bibitem{Lin_MAMS90}
Tom Lindstr{\o}m, \emph{Brownian motion on nested fractals}, Mem. Amer. Math.
  Soc. \textbf{83} (1990), no.~420, iv+128. \MR{MR988082 (90k:60157)}

\bibitem{NeStTe_04}
Jonathan Needleman, Robert~S. Strichartz, Alexander Teplyaev, and Po-Lam Yung,
  \emph{Calculus on the {S}ierpinski gasket. {I}. {P}olynomials, exponentials
  and power series}, J. Funct. Anal. \textbf{215} (2004), no.~2, 290--340.
  \MR{MR2150975 (2006h:28013)}

\bibitem{Rog_08}
Luke~G. Rogers, \emph{Estimates for the resolvent kernel of the {L}aplacian on
  p.c.f. self-similar fractals and blowups}, Trans. Amer. Math. Soc.
  \textbf{364} (2012), no.~3, 1633--1685. \MR{2869187}

\bibitem{Sab:JFA00}
Christophe Sabot, \emph{Pure point spectrum for the {L}aplacian on unbounded
  nested fractals}, J. Funct. Anal. \textbf{173} (2000), no.~2, 497--524.
  \MR{MR1760624 (2001j:35216)}

\bibitem{See_67}
R.~T. Seeley, \emph{Complex powers of an elliptic operator}, Singular
  {I}ntegrals ({P}roc. {S}ympos. {P}ure {M}ath., {C}hicago, {I}ll., 1966),
  Amer. Math. Soc., Providence, R.I., 1967, pp.~288--307. \MR{MR0237943 (38
  \#6220)}

\bibitem{See_69}
\bysame, \emph{Analytic extension of the trace associated with elliptic
  boundary problems}, Amer. J. Math. \textbf{91} (1969), 963--983.
  \MR{MR0265968 (42 \#877)}

\bibitem{sikora-2008}
Adam Sikora, \emph{Multivariable spectral multipliers and analysis of
  quasielliptic operators on fractals}, Indiana Univ. Math. J. \textbf{58}
  (2009), no.~1, 317--334. \MR{MR2504414}

\bibitem{Ste_PMS30_70}
Elias~M. Stein, \emph{Singular integrals and differentiability properties of
  functions}, Princeton Mathematical Series, No. 30, Princeton University
  Press, Princeton, N.J., 1970. \MR{MR0290095 (44 \#7280)}

\bibitem{Ste_AMS63_70}
\bysame, \emph{Topics in harmonic analysis related to the {L}ittlewood-{P}aley
  theory.}, Annals of Mathematics Studies, No. 63, Princeton University Press,
  Princeton, N.J., 1970. \MR{MR0252961 (40 \#6176)}

\bibitem{Ste_PMS43_93}
\bysame, \emph{Harmonic analysis: real-variable methods, orthogonality, and
  oscillatory integrals}, Princeton Mathematical Series, vol.~43, Princeton
  University Press, Princeton, NJ, 1993, With the assistance of Timothy S.
  Murphy, Monographs in Harmonic Analysis, III. \MR{MR1232192 (95c:42002)}

\bibitem{StSha_03}
Elias~M. Stein and Rami Shakarchi, \emph{Complex analysis}, Princeton Lectures
  in Analysis, II, Princeton University Press, Princeton, NJ, 2003.
  \MR{MR1976398 (2004d:30002)}

\bibitem{Str_CJM98}
Robert~S. Strichartz, \emph{Fractals in the large}, Canad. J. Math. \textbf{50}
  (1998), no.~3, 638--657. \MR{MR1629847 (99f:28015)}

\bibitem{Str_TAMS03}
\bysame, \emph{Fractafolds based on the {S}ierpi\'nski gasket and their
  spectra}, Trans. Amer. Math. Soc. \textbf{355} (2003), no.~10, 4019--4043
  (electronic). \MR{MR1990573 (2004b:28013)}

\bibitem{Str_JFA03}
\bysame, \emph{Function spaces on fractals}, J. Funct. Anal. \textbf{198}
  (2003), no.~1, 43--83. \MR{MR1962353 (2003m:46058)}

\bibitem{Str:TAM04}
\bysame, \emph{Analysis on products of fractals}, Trans. Amer. Math. Soc.
  \textbf{357} (2005), no.~2, 571--615 (electronic). \MR{MR2095624
  (2005m:31016)}

\bibitem{Str_Prin06}
\bysame, \emph{Differential equations on fractals}, Princeton University Press,
  Princeton, NJ, 2006, A tutorial. \MR{MR2246975 (2007f:35003)}

\bibitem{Str_08Quant}
\bysame, \emph{A fractal quantum mechanical model with {C}oulomb potential},
  Commun. Pure Appl. Anal. \textbf{8} (2009), no.~2, 743--755. \MR{MR2461574}

\bibitem{Tay_PMS81}
Michael~E. Taylor, \emph{Pseudodifferential operators}, Princeton Mathematical
  Series, vol.~34, Princeton University Press, Princeton, N.J., 1981.
  \MR{MR618463 (82i:35172)}

\bibitem{Tep_JFA98}
Alexander Teplyaev, \emph{Spectral analysis on infinite {S}ierpi\'nski
  gaskets}, J. Funct. Anal. \textbf{159} (1998), no.~2, 537--567. \MR{MR1658094
  (99j:35153)}

\bibitem{DuOuSi_JFA02}
Xuan Thinh~Duong, El~Maati Ouhabaz, and Adam Sikora, \emph{Plancherel-type
  estimates and sharp spectral multipliers}, J. Funct. Anal. \textbf{196}
  (2002), no.~2, 443--485. \MR{MR1943098 (2003k:43012)}

\end{thebibliography}

\end{document}